\documentclass{amsart}

\usepackage[margin=1in]{geometry}
\usepackage{amsmath,amsthm,amssymb,amstext}
\usepackage{mathrsfs}
\usepackage{tikz-cd}
\usepackage{hyperref}
\usepackage[alphabetic,backrefs]{amsrefs}

\newtheorem{theorem}{Theorem}
\newtheorem{letteredtheorem}{Theorem}

\newtheorem{letteredproposition}[letteredtheorem]{Proposition}

\newtheorem{proposition}[theorem]{Proposition}
\newtheorem{lemma}[theorem]{Lemma}
\newtheorem{corollary}[theorem]{Corollary}

\theoremstyle{remark}
\newtheorem{remark}[theorem]{Remark}

\theoremstyle{definition}
\newtheorem{definition}[theorem]{Definition}

\numberwithin{theorem}{section}

\hypersetup{
  colorlinks=true,
  linkcolor=blue,
  anchorcolor=blue,
  citecolor=blue
}

\hypersetup{
  pdftitle={Stability of Kernel Sheaves Associated to Rank One Torsion-Free Sheaves},
  pdfauthor={Nick Rekuski}
  }
  
\begin{document}

  \title{
    Stability of Kernel Sheaves Associated to Rank One Torsion-Free Sheaves
  }
  \author{
    Nick Rekuski
  }
  \date{
    May 12, 2023
  }
  \address{
    Department of Mathematics, Wayne State University, Detroit, MI 48202, USA
  }
  \email{
    rekuski@wayne.edu
  }
  \thanks{
    The author was partially supported by the NSF grant DMS-2101761 during preparation of this article.
    The author is supported by the U.S. Department of Energy, Office of Science, Basic Energy Sciences, under Award Number DE-SC-SC0022134.
    This work is also partially supported by an OVPR Postdoctoral Award at Wayne State University.
  }

  \begin{abstract}
    We show the kernel sheaf associated to a sufficiently positive torsion-free sheaf of rank $1$ is slope stable.
    Furthermore, we are able to give an explicit bound for ``sufficiently positive."
    This settles a conjecture of Ein-Lazarsfeld-Mustopa.
    The main technical lemma is a bound on the number of global sections of a torsion-free, globally generated sheaf in terms of its rank, degree, and invariants of the variety.
  \end{abstract}

  \maketitle

\section{Introduction}
  It is difficult to construct slope stable bundles with given topological invariants on higher dimensional varieties.
  For example, it is completely open whether there exists a slope stable, rank $2$ bundle on $\mathbb{P}^{7}$.
  This difficulty, in part, is because categorical constructions involving slope stable sheaves do not produce slope stable sheaves.
  In this article, we consider the explicit case of the kernel of the natural surjection
  \[
    0\to\mathscr{M}_{\mathscr{L}}\to H^{0}(\mathscr{L})\otimes\mathscr{O}_{X}\to\mathscr{L}\to 0
  \]
  when $\mathscr{L}$ is a globally generated, torsion-free sheaf of rank $1$.
  Kernels arising via this construction are called kernel sheaves.
  Over curves, Ein and Lazarsfeld showed $\mathscr{M}_{\mathscr{L}}$ is slope stable as soon as $\operatorname{deg}(\mathscr{L})> 2g$~\cite{ein1989}*{Proposition 3.2}.
  The expectation is that a similar result holds in higher dimensions.
  
  The most general higher dimensional result is that on a smooth projective surface or smooth projective higher dimensional variety of Picard rank $1$ the kernel bundle associated to a sufficiently positive line bundle is slope stable~\cite{ein2013}*{Theorem A and Proposition C}.
  However, this method does not give an explicit bound on sufficiently positive nor does it extend to higher dimensional varieties of higher Picard rank~\cite{ein2013}*{Problem 2.4, Conjecture 2.6}.
  We are able to settle both of these problems:
  
  \begin{letteredtheorem}[\ref{theorem:lazarsfeldMukaiSheavesAreStable}]
    \label{maintheorem}
    Suppose $\mathscr{L}$ is a globally generated, torsion-free sheaf of rank $1$ on a smooth, projective variety $X$ with fixed very ample divisor $H$.
    For ease of notation, let $g$ be the sectional genus of $X$ (with respect to $H$).
    If
    \begin{align*}
      h^{0}(\mathscr{L}) & >\frac{\operatorname{deg}_{H}(\mathscr{L})}{\operatorname{deg}_{H}(\mathscr{L})-1}\Bigg{(}H^{n}\binom{\frac{\operatorname{deg}_{H}(\mathscr{L})-1-(g-1)}{H^{n}}+n-1}{n}-1 \\
      &\quad+\frac{(n-1)(n+g-1)}{n}\binom{\frac{\operatorname{deg}_{H}(\mathscr{L})-1-(2g-2)}{H^{n}}+n-3}{n-2}\binom{\frac{2g-2}{H^{n}}+n-1}{n-1}\Bigg{)}+1
    \end{align*}
    and
    \[
      h^{0}(\mathscr{L})>\frac{\operatorname{deg}_{H}(\mathscr{L})}{2g-2}\left(\left(\frac{2g-2}{2n}+1\right)\binom{\frac{2g-2}{H^{n}}+n-1}{n-1}-1\right)+1
    \]
    then $\mathscr{M}_{\mathscr{L}}$ is $\mu_{H}$-stable.
  \end{letteredtheorem}
  
  A priori, the necessary inequalities of Theorem \ref{maintheorem} may never apply.
  To this end, Corollary \ref{corollary:assymptoticVersion} shows $h^{0}(\mathscr{L}(k))$ satisfies the desired inequalities for $k\gg 0$.
  Furthermore, in Remark \ref{remark:howToGetExplicitBound} we show how to obtain an effective bound on $k\gg 0$ in terms of topological invariants and the Castelnuovo-Mumford regularity of $\mathscr{L}$.
  
  The main technical lemma to prove Theorem \ref{maintheorem} is a bound on the number of global sections of a torsion-free, globally generated sheaf solely in terms of topological invariants of that sheaf.
  We believe this bound is of independent interest.
  
  \begin{letteredproposition}[\ref{proposition:boundOnSections}]
    \label{mainproposition}
    Along with the assumptions of Theorem \ref{maintheorem} also assume $\mathscr{E}$ is a torsion-free, globally generated sheaf.
    If $\operatorname{deg}_{H}(\mathscr{E})\leq 2g-2$ then
    \[
      h^{0}(\mathscr{E})\leq \left(1+\frac{\operatorname{deg}_{H}(\mathscr{E})}{n}\right)\binom{\frac{\operatorname{deg}_{H}(\mathscr{E})}{H^{n}}+n-1}{n-1}+\operatorname{rank}(\mathscr{E})-1.
    \]
    If $\operatorname{deg}_{H}(\mathscr{E})\geq 2g-1$ then
    \begin{align*}
      h^{0}(\mathscr{E})&\leq H^{n}\binom{\frac{\operatorname{deg}_{H}(\mathscr{E})-(g-1)}{H^{n}}+n-1}{n}+\frac{1+g}{2}\binom{\frac{2g-2}{H^{n}}+n-1}{n-1}\binom{\frac{\operatorname{deg}_{H}(\mathscr{E})-(2g-2)}{H^{n}}+n-2}{n-2}+\operatorname{rank}(\mathscr{E})-1.
    \end{align*}
  \end{letteredproposition}
  
  In fact, this result holds for any torsion-free sheaf globally generated outside codimension $2$ (Definition \ref{definition:genericallyGloballyGenerated}  and Remark \ref{remark:extendsToGloballyGeneratedOutsideCodim2}).
  We discuss optimality of Proposition \ref{mainproposition} in Remark \ref{remark:howGoodIsBound}.
  
\subsection*{Outline}
  In section $2$ we recall relevant background regarding slope stable sheaves and kernel sheaves.
  In section $3$ we prove Proposition \ref{mainproposition}.
  In section $4$ we use this bound on global sections to prove Theorem \ref{maintheorem}.
  
  Our proof of Theorem \ref{maintheorem} broadly follow the same argument as \cite{butler1994}*{Theorem 1.2} where Butler shows on a curve that kernel bundles associated to sufficiently positive slope stable bundles are also slope stable.
  However, there are new difficulties in higher dimensions that must be addressed.
  
  Let $\mathscr{N}$ be a maximal  destabilizing subsheaf of $\mathscr{M}_{\mathscr{L}}$.
  Since $\mathscr{M}_{\mathscr{L}}$ is a kernel sheaf, there is an induced short exact sequence
  \[
    0\to\mathscr{N}\to\mathscr{O}_{X}^{\oplus h^{0}(\mathscr{L})}\to\mathscr{C}\to 0.
  \]
  Using this short exact sequence, we find $\operatorname{rank}(\mathscr{N})\leq h^{0}(\mathscr{C})-\operatorname{rank}(\mathscr{C})$.
  On curves, Butler bounds $h^{0}(\mathscr{C})-\operatorname{rank}(\mathscr{C})$ in terms of $\mu(\mathscr{C})$ using \cite{butler1994}*{Lemma 1.10} and the Riemann-Roch theorem.
  Proposition \ref{mainproposition} is our higher dimensional analogue.
  Butler then uses his bound on $\operatorname{rank}(\mathscr{N})$ to bound $\mu(\mathscr{N})$ in terms of invariants of $\mathscr{L}$~\cite{butler1994}*{Proposition 1.4}.
  For higher dimensions, our analogue is Lemma \ref{lemma:boundsAreIncreasing}. 
  Theorem \ref{maintheorem} almost immediately follows.
  
\subsection*{Notation and Assumptions}
  Suppose $X$ is a smooth, projective variety (i.e. a smooth, integral, projective scheme of finite type over an algebraically closed field) of dimension $\operatorname{dim}(X)=n$.
  Note we allow arbitrary characteristic of the base field.
  Fix a very ample divisor $H$ on $X$.
  We denote the sectional genus of $X$ with respect to $H$ by $g$.
  In other words, $g$ is the genus of the smooth, integral curve $H^{n-1}$.
  By the adjunction formula, we can rewrite the sectional genus as
  \begin{equation}
    \label{equation:sectionalGenusByAdjunction}
    g=1+\frac{n-1}{2}H^{n}-\frac{\operatorname{c}_{1}(X)\cdot H^{n-1}}{2}
  \end{equation}
  
  We use script letters (e.g. $\mathscr{F}, \mathscr{E}, \mathscr{G}$) to denote coherent sheaves on $X$.
  We reserve script L ($\mathscr{L}$) for torsion-free sheaves of rank $1$.
  The degree of $\mathscr{E}$ (with respect to $H$) is $\operatorname{deg}_{H}(\mathscr{E})=H^{n-1}\cdot\operatorname{c}_{1}(\mathscr{E})$ where $c_{1}(\mathscr{E})$ is the first Chern class of $\mathscr{E}$ viewed as a divisor on $X$.
  If $n=\operatorname{dim}(X)=1$ then $\operatorname{deg}_{H}(\mathscr{E})$ is independent of $H$ so we drop $H$ from the notation $\operatorname{deg}(\mathscr{E})=\operatorname{deg}_{H}(\mathscr{E})$.
  We also write $\operatorname{codim}(\mathscr{E})=\operatorname{dim}(X)-\operatorname{dim}\operatorname{Supp}(\mathscr{E})$ and $\operatorname{Sing}(\mathscr{E})$ for the closed subscheme of $X$ consisting of stalks where $\mathscr{E}$ is not free.
  Recall if $\mathscr{E}$ is torsion-free then $\operatorname{codim}\operatorname{Sing}(\mathscr{E})\geq 2$.
  
  If $x$ is a real number and $n$ is a nonnegative integer we set
  \begin{equation}
    \label{equation:binomialCoefficient}
    \binom{x+n}{n}=
    \begin{cases}
      \displaystyle \frac{(x+n)(x+n-1)\cdots(x+1)}{n!}: & x\geq 0, n\geq 1 \\
      0: & x<0, n\geq 1 \\
      1: & n=0 \\
    \end{cases}.
  \end{equation}
  If $x$ is a nonnegative integer and $n\geq 1$ then $\binom{x+n}{n}$ agrees with the usual definition of the binomial coefficient.
  
  \subsection*{Acknowledgments}
  The author is thankful to Rajesh Kulkarni and Yusuf Mustopa for many useful discussions.
  The author is also thankful to Federico Caucci, Peter Newstead, and Shitan Xu for comments on an earlier draft of this paper.
  
\section{Generalities on Slope Stability and Kernel Sheaves}
  In this section we recall slope stability and kernel sheaves.
  We then discuss a brief history of the stability of kernel sheaves.
  The definitions and results of this section are well known.
  
  \begin{definition}
    \label{definition:slopeStableSheaf}
    Let $X$ be a smooth, projective variety equipped with ample divisor $H$.
    For any nonzero coherent sheaf $\mathscr{E}$ we define its slope (with respect to $H$) to be
    \[
      \mu_{H}(\mathscr{E})=
      \begin{cases}
        \frac{\operatorname{deg}_{H}(\mathscr{E})}{\operatorname{rank}(\mathscr{E})}: & \operatorname{rank}(\mathscr{E})\neq 0 \\
        +\infty: & \operatorname{rank}(\mathscr{E})=0
      \end{cases}.
    \]
    If $\operatorname{dim}(X)=1$ then we drop $H$ from the notation: $\mu_{H}=\mu$.
    
    We say a nonzero torsion-free sheaf $\mathscr{E}$ is $\mu_{H}$-(semi)stable if every subsheaf $0\to\mathscr{F}\to\mathscr{E}$ satisfying $0<\operatorname{rank}(\mathscr{F})<\operatorname{rank}(\mathscr{E})$ also satisfies $\mu_{H}(\mathscr{F})(\leq)<\mu_{H}(\mathscr{E})$.
    
    The quantity $\mu_{H}(\mathscr{E})$ is called the slope of $\mathscr{E}$ (with respect to $H$) and so $\mu_{H}$-stability is also often called sloe stability.
  \end{definition}
  
  As is well known, it suffices to only consider saturated subsheaves.
  
  \begin{lemma}
    \label{lemma:justLookAtSaturated}
    With the assumptions of Definition \ref{definition:slopeStableSheaf}, the following are equivalent.
    \begin{enumerate}
      \item{
        $\mathscr{E}$ is $\mu_{H}$-(semi)stable.
      }
      \item{
        If $0\to\mathscr{F}\to\mathscr{E}$ is a proper, nonzero subsheaf such that $\mathscr{E}/\mathscr{F}$ is torsion-free then $\mu_{H}(\mathscr{F}){(\leq)}<\mu_{H}(\mathscr{E})$.
      }
    \end{enumerate}
  \end{lemma}

  We also note the slope is well-behaved in short exact sequences of sheaves supported everywhere.
  
  \begin{lemma}[Seesaw Inequality, \cite{rudakov1997}*{Lemma 3.2}]
    \label{lemma:seesawInequality}
    Suppose $0\to\mathscr{F}\to\mathscr{E}\to\mathscr{G}\to 0$ is a short exact sequence of coherent sheaves.
    If $\operatorname{rank}(\mathscr{F}),\operatorname{rank}(\mathscr{E}),\operatorname{rank}(\mathscr{G})\neq 0$ then one of the following inequalities must hold:
    \begin{itemize}
      \item{
        $\mu_{H}(\mathscr{F})<\mu_{H}(\mathscr{E})<\mu_{H}(\mathscr{G})$,
      }
      \item{
        $\mu_{H}(\mathscr{F})=\mu_{H}(\mathscr{E})=\mu_{H}(\mathscr{G})$, or
      }
      \item{
        $\mu_{H}(\mathscr{F})>\mu_{H}(\mathscr{E})>\mu_{H}(\mathscr{G})$.
      }
    \end{itemize}
  \end{lemma}
  
  We now recall kernel sheaves.
  
  \begin{definition}
    \label{definition:lazarsfeldMukaiSheaf}
    Suppose $\mathscr{E}$ is a globally generated, torsion-free sheaf on $X$.
    Therefore, there is a short exact sequence
    \[
      0\to\mathscr{M}_{\mathscr{E}}\to H^{0}(\mathscr{E})\otimes\mathscr{O}_{X}\to\mathscr{E}\to 0.
    \]
    whose kernel is called the kernel sheaf associated to $\mathscr{E}$.
    Kernel sheaves are also called syzygy, Lazarsfeld-Mukai, or Lazarsfeld sheaves.
    
    If $\mathscr{E}$ is clear from context, we will often drop the subscript: $\mathscr{M}_{\mathscr{E}}=\mathscr{M}$.
    Moreover, if $\mathscr{M}_{\mathscr{E}}$ is locally free, then we say $\mathscr{M}_{\mathscr{E}}$ is a kernel bundle (rather than a kernel sheaf).
  \end{definition}
  
  Over curves, $\mu$-stability of $\mathscr{M}_{\mathscr{E}}$ is generally well understood.
  For the following suppose $C$ is a smooth curve of genus $g$.
  \begin{itemize}
    \item{
      If $\mathscr{L}$ is a line bundle on $C$ satisfying $\operatorname{deg}(\mathscr{L})>2g$ then $\mathscr{M}_{\mathscr{L}}$ is $\mu_{H}$-stable~\cite{ein1989}*{Proposition 3.2}.
    }
    \item{
      If $\mathscr{E}$ is a $\mu$-stable bundle on $C$ satisfying $\mu(\mathscr{E})>2g$ then $\mathscr{M}_{\mathscr{E}}$ is $\mu$-stable~\cite{butler1994}*{Theorem 1.2}.
    }
  \end{itemize}
  There are also results improving the necessary bounds in the above results~\cites{butler1997,camere2008}, and more recent results have considered the case where $\mathscr{E}$ is generated by an incomplete linear system~\cites{bhosle2015,brambila2019}.
  
  In higher dimensions, $\mu_{H}$-stability of kernel sheaves is less understood.
  Results tend to be for specific classes of varieties or do not give effective bounds on positivity.
  \begin{itemize}
    \item{
      Assume the base field is of characteristic $0$.
      If $d\geq 0$ then the kernel bundle associated to $\mathscr{O}_{\mathbb{P}^{n}}(d)$ is $\mu_{H}$-semistable~\cite{flenner1984}*{Corollary 2.2}. 
    }
    \item{
      Suppose $\mathscr{L}$ is the image of $V\otimes\mathscr{O}_{X}\to\mathscr{O}_{\mathbb{P}^{n}}(d)$ for some subspace $V\subseteq H^{0}(\mathscr{O}_{X}(d))$.
      If
      \[
        \operatorname{dim}(V)> \binom{d-1+n}{n}+\frac{1}{d}\binom{d-1+n}{n}-\frac{1}{d}
      \]
      then $\mathscr{M}_{\mathscr{L}}$ is $\mu_{H}$-stable~\cite{coanda2011}*{Theorem 1}. 
    }
    \item{
      If $X\subseteq\mathbb{P}^{n}$ is a complete intersection of multidegree $(d,d,\ldots,d)$ then the kernel bundle associated to $\mathscr{O}_{X}(d)$ is $\mu_{H}$-stable~\cite{coanda2011}*{Proposition 2}.
    }
    \item{
      Assume $X$ is an abelian (resp. $K3$) surface over $\mathbb{C}$.
      If $\mathscr{L}$ is a globally generated, ample line bundle on $X$ (resp. satisfying $\mathscr{L}^{2}\geq 14$) then $\mathscr{M}_{\mathscr{L}}$ is $\mu_{\mathscr{L}}$-stable~\cite{camere2012}*{Theorem 1, Theorem 2}.
    }
    \item{
      Suppose $X$ is a surface (resp. $\operatorname{dim}(X)\geq 3$ and $\operatorname{Pic}(X)=\mathbb{Z}$).
      Assume $\mathscr{L}=\mathscr{O}_{X}(dH+\delta)$ where $H\cdot\delta=0$ and $\delta^{2}\leq 0$ (resp. $\mathscr{L}=\mathscr{O}_{X}(dH)$).
      If $d\gg 0$ then $\mathscr{M}_{\mathscr{L}}$ is $\mu_{H}$-stable~\cite{ein2013}*{Theorem A, Theorem B}.
    }
    \item{
      Assume $X$ is an abelian variety.
      If $\mathscr{L}$ is an ample line bundle on $X$ then $\mathscr{M}_{\mathscr{L}^{\otimes d}}$ is $\mu_{\mathscr{L}}$-semistable for all $d\geq 2$.
      Furthermore, if $X$ is simple (i.e. $X$ contains no non-trivial abelian subvarieties) $\mathscr{L}$ is $\mu_{\mathscr{L}}$-stable~\cite{caucci2021}*{Theorem 1}.
    }
    \item{
      Suppose $X$ is an Enriques (resp. bielliptic) surface over a field of characteristic $\neq 2$ (resp. $\neq 2,3$).
      If $\mathscr{L}$ is a globally generated, ample line bundle on $X$ then $\mathscr{M}_{\mathscr{L}}$ is $\mu_{\mathscr{L}}$-stable~\cite{mukherjee2022}*{Theorem 3.5}.
    }
    \item{
      Assume $X$ is a Del Pezzo or Hirzebruch surface.
      If $\mathscr{L}$ is globally generated and ample then $\mathscr{M}_{\mathscr{L}}$ is $\mu_{\mathscr{L}}$-stable~\cite{torres2022}*{Corollary 3.3, Corollary 3.4}.
    }
  \end{itemize}
  
  The proofs of the $\operatorname{dim}(X)\geq 2$ results broadly fall into two techniques.
  The first technique, due to Coand\u{a}~\cite{coanda2011}, is to use Green's vanishing theorem \cite{green1984}*{3.a.1} to show kernel sheaves are cohomologically stable---which implies $\mu_{H}$-stable.
  The second technique, due to Camere~\cite{camere2012}, is to restrict the problem to curves and analyze the short exact sequence
  \[
    0\to\mathscr{O}^{\oplus k}\to\mathscr{M}_{\mathscr{L}}|_{H}\to\mathscr{M}_{\mathscr{L}|_{H}}\to 0
  \]
  noting that $\mathscr{M}_{\mathscr{L}|_{H}}$ is $\mu_{H}$-stable by \cite{ein1989}*{Proposition 3.2}. 
  As described in the outline, our method for proving $\mu_{H}$-stability of kernel sheaves is closer to \cite{butler1994}*{Theorem 1.2} rather than either of the techniques discussed above.
  
\section{Bounding Global Sections}
  In this section we bound the number of global sections of a globally generated sheaf in terms of its rank and degree.

  The following binomial identity is well known and can be proven via induction on $m-a$.
  The corresponding weaker inequality is used extensively in the proof of Lemma \ref{lemma:boundGlobalSectionsLineBundle}.
  See Equation \ref{equation:binomialCoefficient} for our convention for the binomial coefficient.
  
  \begin{lemma}
    \label{lemma:binomialIdentities}
    If $x$ is a real number and $a,k,m$ are positive integers satisfying $x-m-k\geq 0$ then 
    \[
      \sum_{i=a}^{m}\binom{x-i}{k}=\binom{x-a+1}{k+1}-\binom{x-m}{k+1}.
    \]
    In particular, we have the inequality
    \[
      \sum_{i=a}^{m}\binom{x-i}{k}\leq\binom{x-a+1}{k+1}.
    \]      
  \end{lemma}
    
  We first bound the global sections of a torsion-free sheaf of rank $1$ (not necessarily globally generated).
  This result can be thought of as a higher dimensional generalization of Clifford's Theorem.
  The bound when $\operatorname{dim}(X)=n=1$ is classical.
  The bound for $n\geq 2$ is seemingly new, but the argument is similar to \cite{langer2004}*{Theorem 3.3}.
  The bound in the case $\operatorname{deg}_{H}(\mathscr{L})\geq 2g-1$ seems especially complicated, but it is written in this way to lend itself to induction on $n$---in Proposition \ref{proposition:boundOnSections} we simplify this bound.
  
  \begin{lemma}
    \label{lemma:boundGlobalSectionsLineBundle}
    Suppose $X$ is a smooth, projective variety equipped with very ample divisor $H$.
    For ease of notation, let $g$ be the sectional genus of $X$ (with respect to $H$).
    Furthermore, assume $\mathscr{L}$ is a torsion-free sheaf of rank $1$ on $X$.
    If $\operatorname{deg}_{H}(\mathscr{L})\leq 2g-2$ then
    \[
      h^{0}(\mathscr{L})\leq \frac{H^{n}}{2}\binom{\frac{\operatorname{deg}_{H}(\mathscr{L})}{H^{n}}+n-1}{n}+\binom{\frac{\operatorname{deg}_{H}(\mathscr{L})}{H^{n}}+n-1}{n-1}.
    \]
    If $\operatorname{deg}_{H}(\mathscr{L})\geq 2g-1$ then
    \begin{align*}
      h^{0}(\mathscr{L})&\leq H^{n}\binom{\frac{\operatorname{deg}_{H}(\mathscr{L})-(g-1)}{H^{n}}+n-1}{n} \\
      &\qquad+\sum_{i=0}^{n-2}\frac{n-i+g-1}{n-i}\binom{\frac{\operatorname{deg}_{H}(\mathscr{L})-(2g-2)}{H^{n}}+i-1}{i}\binom{\frac{2g-2}{H^{n}}+n-1-i}{n-1-i}.
    \end{align*}
  \end{lemma}
  \begin{proof}
    We proceed by induction on $n=\operatorname{dim}(X)$.
    If $\operatorname{dim}(X)=1$ then the $\operatorname{deg}(\mathscr{L})\leq 2g-2$ bound follows from Clifford's theorem.
    If $\operatorname{deg}(\mathscr{L})\geq 2g-1$ then, by Serre duality, $h^{1}(\mathscr{L})=0$ and so the result follows by the Riemann-Roch theorem.
    
    We proceed with the inductive step.
    Since $\mathscr{L}$ is torsion-free, for general hyperplane $H$, we have the short exact sequence
    \[
      0\to\mathscr{L}(-1)\to\mathscr{L}\to\mathscr{L}|_{H}\to 0.
    \]
    so $h^{0}(\mathscr{L})\leq h^{0}(\mathscr{L}|_{H})+h^{0}(\mathscr{L}(-1))$.
    Continuing this process on $h^{0}(\mathscr{L}(-1))$ gives
    \[
      h^{0}(\mathscr{L})\leq\sum_{i=0}^{\left\lfloor\frac{\operatorname{deg}_{H}(\mathscr{L})}{H^n}\right\rfloor}h^{0}(\mathscr{L}(-i)|_{H}).
    \]
    where $\lfloor x\rfloor=\max\{n\in\mathbb{Z}\mid n\leq x\}$.   
        
    If $\operatorname{deg}_{H}(\mathscr{L})\leq 2g-2$, by the inductive hypothesis and Lemma \ref{lemma:binomialIdentities}, we find
    \begin{align*}
      \sum_{i=0}^{\lfloor\frac{\operatorname{deg}_{H}(\mathscr{L})}{H^{n}}\rfloor}h^{0}(\mathscr{L}(-i)|_{H}) & \leq \sum_{i=0}^{\lfloor\frac{\operatorname{deg}_{H}(\mathscr{L})}{H^{n}}\rfloor} \left(\frac{H^{n}}{2}\binom{\frac{\operatorname{deg}_{H}(\mathscr{L})}{H^{n}}-i+n-2}{n-1}+\binom{\frac{\operatorname{deg}_{H}(\mathscr{L})}{H^{n}}-i+n-2}{n-2}\right) \\
      & \leq \frac{H^{n}}{2}\binom{\frac{\operatorname{deg}_{H}(\mathscr{L})}{H^{n}}+n-1}{n}+\binom{\frac{\operatorname{deg}_{H}(\mathscr{L})}{H^{n}}+n-1}{n-1}
    \end{align*}
    as desired.
    
    If $\operatorname{deg}_{H}(\mathscr{L})\geq 2g-1$ then, by the same argument as above,
    \begin{align*}
      h^{0}(\mathscr{L}) & \leq \sum_{i=0}^{\lfloor \frac{\operatorname{deg}_{H}(\mathscr{L})}{H^{n}}\rfloor} h^{0}(\mathscr{L}(-i)|_{H}) \\
      & = \sum_{i=0}^{\lceil\frac{\operatorname{deg}_{H}(\mathscr{L})-(2g-2)}{H^{n}}\rceil-1}h^{0}(\mathscr{L}(-i)|_{H})+\sum_{i=\lceil\frac{\operatorname{deg}_{H}(\mathscr{L})-(2g-2)}{H^{n}}\rceil}^{\lfloor\frac{\operatorname{deg}_{H}(\mathscr{L})}{H^{n}}\rfloor}h^{0}(\mathscr{L}(-i)|_{H})
    \end{align*}
    where $\lceil x\rceil=\min\{n\in\mathbb{Z}\mid n\geq x\}$.
    We consider each summand separately.
        
    In the first summand, $i\leq \lceil\frac{\operatorname{deg}_{H}(\mathscr{L})-(2g-2)}{H^{n}}\rceil-1$ so $\operatorname{deg}_{H}(\mathscr{L}(-i)|_{H})>2g-2$.
    Therefore, by the inductive hypothesis and Lemma \ref{lemma:binomialIdentities},
    \begin{align*}
      & \sum_{i=0}^{\lceil\frac{\operatorname{deg}_{H}(\mathscr{L})-(2g-2)}{H^{n}}\rceil-1} h^{0}(\mathscr{L}(-i)|_{H}) \\
      & \quad\leq\sum_{i=0}^{\lceil\frac{\operatorname{deg}_{H}(\mathscr{L})-(2g-2)}{H^{n}}\rceil-1} H^{n}\binom{\frac{\operatorname{deg}_{H}(\mathscr{L})-(g-1)}{H^{n}}-i+n-2}{n-1} \\
      & \qquad+ \sum_{i=0}^{\lceil\frac{\operatorname{deg}_{H}(\mathscr{L})-(2g-2)}{H^{n}}\rceil-1} \sum_{j=0}^{n-3}\frac{n-1-j+g-1}{n-1-j}\binom{\frac{\operatorname{deg}_{H}(\mathscr{L})-(2g-2)}{H^{n}}-i+j-1}{j}\binom{\frac{2g-2}{H^{n}}+n-2-j}{n-2-j} \\
      & \quad\leq H^{n}\binom{\frac{\operatorname{deg}_{H}(\mathscr{L})-(g-1)}{H^{n}}+n-1}{n}+ \sum_{j=0}^{n-3}\frac{n-1-j+g-1}{n-1-j}\binom{\frac{\operatorname{deg}_{H}(\mathscr{L})-(2g-2)}{H^{n}}+j}{j+1}\binom{\frac{2g-2}{H^{n}}+n-2-j}{n-2-j}  \\
      & \quad = H^{n}\binom{\frac{\operatorname{deg}_{H}(\mathscr{L})-(g-1)}{H^{n}}+n-1}{n}+ \sum_{j=1}^{n-2}\frac{n-j+g-1}{n-j}\binom{\frac{\operatorname{deg}_{H}(\mathscr{L})-(2g-2)}{H^{n}}+j-1}{j}\binom{\frac{2g-2}{H^{n}}+n-1-j}{n-1-j}
    \end{align*}
        
    We now consider the second suummand.
    We will see that this summand only contributes to the ``j=0" term in the above formula. 
    In the second summand, $i\geq\lceil\frac{\operatorname{deg}_{H}(\mathscr{L})-(2g-2)}{H^{n}}\rceil$ so $\operatorname{deg}_{H}(\mathscr{L}(-i)|_{H})\leq 2g-2$.
   Thus, by the inductive hypothesis
   \begin{align*}
     \sum_{i=\lceil\frac{\operatorname{deg}_{H}(\mathscr{L})-(2g-2)}{H^{n}}\rceil}^{\lfloor\frac{\operatorname{deg}_{H}(\mathscr{L})}{H^{n}}\rfloor}h^{0}(\mathscr{L}(-i)|_{H}) & \leq \sum_{i=\lceil\frac{\operatorname{deg}_{H}(\mathscr{L})-(2g-2)}{H^{n}}\rceil}^{\lfloor\frac{\operatorname{deg}_{H}(\mathscr{L})}{H^{n}}\rfloor} \frac{H^{n}}{2}\binom{\frac{\operatorname{deg}_{H}(\mathscr{L})}{H^{n}}-i+n-2}{n-1}+\binom{\frac{\operatorname{deg}_{H}(\mathscr{L})}{H^{n}}-i+n-2}{n-2} \\
    & \leq \frac{H^{n}}{2}\binom{\frac{\operatorname{deg}_{H}(\mathscr{L})}{H^{n}}-\lceil\frac{\operatorname{deg}_{H}(\mathscr{L})-(2g-2)}{H^{n}}\rceil+n-1}{n} \\
     &\qquad+\binom{\frac{\operatorname{deg}_{H}(\mathscr{L})}{H^{n}}-\lceil\frac{\operatorname{deg}_{H}(\mathscr{L})-(2g-2)}{H^{n}}\rceil+n-1}{n-1}.
    \end{align*}
    where the second inequality follows from Lemma \ref{lemma:binomialIdentities}.
    Since $\binom{x+n-1}{n}$ is increasing for $x>0$, we find
    \begin{align*}
      \sum_{i=\lceil\frac{\operatorname{deg}_{H}(\mathscr{L})-(2g-2)}{H^{n}}\rceil}^{\lfloor\frac{\operatorname{deg}_{H}(\mathscr{L})}{H^{n}}\rfloor}h^{0}(\mathscr{L}(-i)|_{H}) & \leq \frac{H^{n}}{2}\binom{\frac{2g-2}{H^{n}}+n-1}{n}+\binom{\frac{2g-2}{H^{n}}+n-1}{n-1} \\
      & =\frac{n+g-1}{n}\binom{\frac{2g-2}{H^{n}}+n-1}{n-1}.
    \end{align*}        

    Combining our inequalities for each summand gives the claimed bound for $\operatorname{deg}_{H}(\mathscr{L})\geq 2g-1$.
  \end{proof}
  
  We recall the well known result that the existence of an injection $0\to\mathscr{O}_{X}^{\oplus \operatorname{rank}(\mathscr{E})}\to\mathscr{E}$ is equivalent to $\mathscr{E}$ being globally generated outside codimension $1$.
  An immediate corollary is if $\mathscr{E}$ is globally generated then there exists an injection $\mathscr{O}_{X}^{\oplus\operatorname{rank}(\mathscr{E})}\to\mathscr{E}$.
  This injection will allow us to extend Lemma \ref{lemma:boundGlobalSectionsLineBundle} to torsion-free, globally generated sheaves of higher rank.
  
  \begin{definition}
    \label{definition:genericallyGloballyGenerated} 
    Suppose $\mathscr{E}$ is a coherent sheaf on $X$.
    Let $\mathscr{C}$ be the cokernel of the natural morphism $H^{0}(\mathscr{E})\otimes\mathscr{O}_{X}\to\mathscr{E}$.
    We say $\mathscr{E}$ is globally generated outside codimension $d$ if $\operatorname{codim}(\mathscr{C})\geq d$.
  \end{definition}
  
  \begin{lemma}
    \label{lemma:globallyGeneratedGivesInjection}
    Suppose $\mathscr{E}$ is a coherent sheaf on $X$.
    The sheaf $\mathscr{E}$ is globally generated outside codimension $1$ if and only if $\operatorname{rank}(\mathscr{E})$ linearly independent global sections of $\mathscr{E}$ induce an injection $\mathscr{O}_{X}^{\oplus\operatorname{rank}(\mathscr{E})}\to\mathscr{E}$.
  \end{lemma}
  \begin{proof}
    First suppose $\mathscr{E}$ is globally generated outside codimension $1$.
    Since $\mathscr{E}$ is globally generated outside codimension $1$, there is an exact sequence
    \[
      0\to\mathscr{M}\to H^{0}(\mathscr{E})\otimes\mathscr{O}_{X}\to\mathscr{E}\to\mathscr{C}\to 0
    \]
    where $\operatorname{codim}(\mathscr{C})\geq 1$.
    By additivity and positivity of the rank, $h^{0}(\mathscr{E})\geq\operatorname{rank}(\mathscr{E})$.
    In particular, we can choose $\operatorname{rank}(\mathscr{E})$ linearly independent global sections of $\mathscr{E}$.
    With this in mind, for ease of notation, let $f$ denote the composition $\mathscr{O}_{X}^{\oplus\operatorname{rank}(\mathscr{E})}\to H^{0}(\mathscr{E})\otimes\mathscr{O}_{X}\to\mathscr{E}$ where the first morphism is given by choosing $\operatorname{rank}(\mathscr{E})$ linearly independent sections of $\mathscr{E}$.
    Therefore, for all $x\in X$ we have the following exact sequence of coherent $\mathscr{O}_{X,x}$-modules
    \[
      0\to\mathscr{K}er(f)_{x}\to\mathscr{O}_{X,x}^{\oplus\operatorname{rank}(\mathscr{E})}\xrightarrow{f}_{x}\mathscr{E}_{x}\to\mathscr{C}oker(f)_{x}\to 0.
    \]
    By definition, $f_{x}$ is an isomorphism outside $X\setminus(\operatorname{Sing}(\mathscr{E})\cup\operatorname{Supp}(\mathscr{C}oker(f)))$.
    Since $\mathscr{C}oker(f)$ is supported in codimension $1$ and $\operatorname{codim}(\operatorname{Sing}(\mathscr{E}))\geq 1$ (for $\mathscr{E}$ is a coherent sheaf), $f$ is an isomorphism outside codimension $\geq 1$.
    In other words, $\mathscr{K}er(f)$ is supported in codimension $\geq 1$.
    Since $\mathscr{O}_{X}^{\operatorname{rank}(\mathscr{E})}$ is torsion-free, $\mathscr{K}er(f)$ is either torsion-free or $0$.
    Thus, we must have $\mathscr{K}er(f)=0$ as claimed.
    Hence, $\mathscr{O}_{X}^{\oplus\operatorname{rank}(\mathscr{E})}\to\mathscr{E}$ is injective, as claimed.
        
    For the converse, assume there is an injection $f:\mathscr{O}_{X}^{\oplus\operatorname{rank}(\mathscr{E})}\to\mathscr{E}$.
    The morphism $f$ must factor through the natural morphism $H^{0}(\mathscr{E})\otimes\mathscr{O}_{X}\to\mathscr{E}$.
    Therefore, by the universal property of the cokernel
    \[
      \mathscr{C}oker(H^{0}(\mathscr{E})\otimes\mathscr{O}_{X}\to\mathscr{E})\subseteq\mathscr{C}oker(f).
    \]
    By additivity of rank, $\operatorname{rank}(\mathscr{C}oker(f))=0$ so $\operatorname{codim}(\mathscr{C}oker(f))\geq 1$---which implies
    \[
      \operatorname{codim}(\mathscr{C}oker(H^{0}(\mathscr{E})\otimes\mathscr{O}_{X}\to\mathscr{E}))\geq 1
    \]
    as well.
  \end{proof}
  
  We now bound the global sections of a globally generated, torsion-free sheaf.
  In fact, our bound holds for sheaves globally generated outside codimension $2$.
  At the same time we simplify the bounds from Lemma \ref{lemma:boundGlobalSectionsLineBundle} so that they lend themselves to Lemma \ref{lemma:boundsAreIncreasing}.
  Later we will see Proposition \ref{proposition:boundOnSections} is optimal (Remark \ref{remark:howGoodIsBound}).
  
  \begin{proposition}
    \label{proposition:boundOnSections}
    Assume $X$ is a smooth, projective variety equipped with very ample divisor $H$.
    Let $g$ be the genus of the smooth integral curve $H^{n-1}$.
    Assume $\mathscr{E}$ is a torsion-free sheaf globally generated outside codimension $2$ on $X$ with $\operatorname{rank}(\mathscr{E})\geq 2$.
    If $\operatorname{deg}_{H}(\mathscr{E})\leq 2g-2$ then
    \[
      h^{0}(\mathscr{E})\leq \left(\frac{\operatorname{deg}_{H}(\mathscr{E})}{2n}+1\right)\binom{\frac{\operatorname{deg}_{H}(\mathscr{E})}{H^{n}}+n-1}{n-1}+\operatorname{rank}(\mathscr{E})-1.
    \]
    If $\operatorname{deg}_{H}(\mathscr{E})\geq 2g-1$ then
    \begin{align*}
      h^{0}(\mathscr{E})& \leq H^{n}\binom{\frac{\operatorname{deg}_{H}(\mathscr{E})-(g-1)}{H^{n}}+n-1}{n}+\operatorname{rank}(\mathscr{E})-1 \\
      & \qquad+ \frac{(n-1)(n+g-1)}{n}\binom{\frac{\operatorname{deg}_{H}(\mathscr{E})-(2g-2)}{H^{n}}+n-3}{n-2}\binom{\frac{2g-2}{H^{n}}+n-1}{n-1}
    \end{align*}
  \end{proposition}
  \begin{proof}
    Since $\mathscr{E}$ is torsion-free and globally generated outside codimension $2$, by Lemma \ref{lemma:globallyGeneratedGivesInjection}, we can choose $\operatorname{rank}(\mathscr{E})-1$ general global sections to obtain the following short exact sequence
    \begin{equation}
      \label{equation:rMinusOneSections}
      0\to\mathscr{O}_{X}^{\oplus\operatorname{rank}(\mathscr{E})-1}\to\mathscr{E}\to\mathscr{L}\to 0.
    \end{equation}
    We claim $\mathscr{L}$ is torsion-free.
    As an aside, if $\mathscr{E}$ is globally generated and reflexive this claim is well known (e.g. \cite{huybrechts2010}*{Example 5.0.1 (2)} for surfaces or \cite{okonek1982}*{Discussion after Definition 2.1} for higher dimensions).
        
    Since $\mathscr{O}_{X}$ is free and $\mathscr{E}$ is torsion-free, by \cite{huybrechts2010}*{Proposition 1.1.10},
    \[
      \operatorname{codim}(\mathscr{E}xt^{q}(\mathscr{L},\mathscr{O}_{X}))=\operatorname{codim}(\mathscr{E}xt^{q}(\mathscr{E},\mathscr{O}_{X}))\geq q+1
    \]
    for all $q\geq 2$.     
    It remains to show $\operatorname{codim}(\mathscr{E}xt^{1}(\mathscr{L},\mathscr{O}_{X}))\geq 2$.
    
    For ease of notation, write $\mathscr{C}$ be the cokernel of the natural morphism $H^{0}(\mathscr{E})\otimes\mathscr{O}_{X}\to\mathscr{E}$.
    Since $\mathscr{E}$ is globally generated outside codimension $2$, by definition, $\operatorname{codim}(\mathscr{C})\geq 2$.
    Set
    \[
      D_{2}=\{x\in X\setminus(\operatorname{Sing}(\mathscr{E})\cup\operatorname{Supp}(\mathscr{C})\mid \operatorname{dim}_{k}\operatorname{span}\{s_{1}(x),s_{2}(x),\ldots,s_{\operatorname{rank}(\mathscr{E})-1}\}\leq \operatorname{rank}(\mathscr{E})-2\}
    \]
    where $s_{1},\ldots,s_{\operatorname{rank}(\mathscr{E})-1}$ are the linearly independent sections.
    By construction, $\mathscr{L}$ is free of rank $1$ on each stalk of $X\setminus D_{2}$. 
    By \cite{kleiman1974}*{Remark 6}, $D_{2}$ has codimension $\operatorname{codim}(D_{2})\geq 2$ in $X\setminus(\operatorname{Sing}(\mathscr{E})\cup\operatorname{Supp}(\mathscr{C}))$.
    Moreover, since $\mathscr{E}$ is torsion-free and globally generated outside codimension $2$, $\operatorname{codim}(\operatorname{Sing}(\mathscr{E})\cup\operatorname{Supp}(\mathscr{C}))\geq 2$.
    Therefore, we find $\operatorname{codim}(\mathscr{E}xt^{1}(\mathscr{L},\mathscr{O}_{X}))\geq 2$.
    Since $\operatorname{codim}(\mathscr{E}xt^{q}(\mathscr{L},\mathscr{O}_{X}))\geq q+1$ for all $q\geq 1$, by \cite{huybrechts2010}*{Proposition 1.1.10}, $\mathscr{L}$ is torsion-free.
        
    Continuing with the short exact sequence (\ref{equation:rMinusOneSections}), we find
    \[
      h^{0}(\mathscr{E})\leq\operatorname{rank}(\mathscr{E})-1+h^{0}(\mathscr{L}).
    \]
    Moreover, by additivity, $\operatorname{deg}_{H}(\mathscr{L})=\operatorname{deg}_{H}(\mathscr{E})$.
    Thus, since $\mathscr{L}$ is torsion-free of rank $1$, $h^{0}(\mathscr{L})$ is bounded by the quantity in Lemma \ref{lemma:boundGlobalSectionsLineBundle}.
    The remainder of the argument involves simplifying these quantities.
    
    If $\operatorname{deg}_{H}(\mathscr{E})\leq 2g-2$, by Lemma \ref{lemma:boundGlobalSectionsLineBundle},
    \begin{align*}
      h^{0}(\mathscr{E}) & \leq \frac{H^{n}}{2}\binom{\frac{\operatorname{deg}_{H}(\mathscr{L})}{H^{n}}+n-1}{n}+\binom{\frac{\operatorname{deg}_{H}(\mathscr{L})}{H^{n}}+n-1}{n-1}+\operatorname{rank}(\mathscr{E})-1 \\
      & = \left(\frac{\operatorname{deg}_{H}(\mathscr{E})}{2n}+1\right)\binom{\frac{\operatorname{deg}_{H}(\mathscr{E})}{H^{n}}+n-1}{n-1}+\operatorname{rank}(\mathscr{E})-1,
    \end{align*}
    as claimed.
    If $\operatorname{deg}_{H}(\mathscr{E})\geq 2g-1$, by Lemma \ref{lemma:boundGlobalSectionsLineBundle},
    \begin{align*}
      h^{0}(\mathscr{E})&\leq H^{n}\binom{\frac{\operatorname{deg}_{H}(\mathscr{E})-(g-1)}{H^{n}}+n-1}{n}+\operatorname{rank}(\mathscr{E})-1 \\
      &\qquad+\sum_{i=0}^{n-2}\frac{n-i+g-1}{n-i}\binom{\frac{\operatorname{deg}_{H}(\mathscr{E})-(2g-2)}{H^{n}}+i-1}{i}\binom{\frac{2g-2}{H^{n}}+n-1-i}{n-1-i} \\
      & \leq H^{n}\binom{\frac{\operatorname{deg}_{H}(\mathscr{E})-(g-1)}{H^{n}}+n-1}{n}+\operatorname{rank}(\mathscr{E})-1 \\
      & \qquad+ \frac{(n-1)(n+g-1)}{n}\binom{\frac{\operatorname{deg}_{H}(\mathscr{E})-(2g-2)}{H^{n}}+n-3}{n-2}\binom{\frac{2g-2}{H^{n}}+n-1}{n-1},
    \end{align*}
    where the second inequality is because $\binom{x+n}{n}$ is increasing in $n$, as desired.
  \end{proof}

  \begin{remark}
    \label{remark:extendsToGloballyGeneratedOutsideCodim2}
    A natural class of examples of torsion-free sheaves that are globally generated outside codimension $2$ but usually not globally generated are reflexive hulls of globally generated, torsion-free sheaves.
    Explicitly, if $\mathscr{E}$ is a torsion-free sheaf globally generated outside codimension $2$ then $\mathscr{E}^{\vee\vee}$ is also globally generated outside codimension $2$.
    To see this, look at the commutative diagram induced from the composition $H^{0}(\mathscr{E})\otimes\mathscr{O}_{X}\to\mathscr{E}\to\mathscr{E}^{\vee\vee}$.
  \end{remark}
  
  We frequently reference the bounds above, so we introduce the following notation.
  
  \begin{definition}
    \label{definition:notationForBounds}
    Fix $n,d\in\mathbb{Z}_{>0}$.
    For ease of notation, we define
    \[
      A_{H}(d)=\left(\frac{d}{2n}+1\right)\binom{\frac{d}{H^{n}}+n-1}{n-1}-1
    \]
    and
    \begin{align*}
      B_{H}(d) & = H^{n}\binom{\frac{d-(g-1)}{H^{n}}+n-1}{n}-1 \\
      & \qquad+\frac{(n-1)(n+g-1)}{n}\binom{\frac{d-(2g-2)}{H^{n}}+n-3}{n-2}\binom{\frac{2g-2}{H^{n}}+n-1}{n-1}
    \end{align*}
  \end{definition}
  Note $A_{H}(\operatorname{deg}_{H}(\mathscr{E}))+\operatorname{rank}(\mathscr{E})$ and $B_{H}(\operatorname{deg}_{H}(\mathscr{E}))+\operatorname{rank}(\mathscr{E})$ are exactly the bounds appearing in Proposition \ref{proposition:boundOnSections}.
  
  We note an asymptotic formula for $B_{H}(d+k H^{n})$ as $k\gg 0$: 
  \begin{remark}
    \label{remark:rewritingBounds}
    As a polynomial in the variable $k$,
    \[
      B_{H}(d+k H^{n})=H^{n} \frac{k^{n}}{n!}+\left(1+d-g+\frac{n-1}{2}H^{n}\right)\frac{k^{n-1}}{(n-1)!}+\cdots.
    \]
    By Equation \ref{equation:sectionalGenusByAdjunction}, it follows that
    \[
      B_{H}(d+k H^{n})=H^{n} \frac{k^{n}}{n!}+\left(d+\frac{c_{1}(X)\cdot H^{n-1}}{2}\right)\frac{k^{n-1}}{(n-1)!}+\cdots.
    \]
  \end{remark}
  
  An easy corollary of Proposition \ref{proposition:boundOnSections} is that a globally generated, torsion-free sheaf of degree $0$ must be trivial.
  This result is standard when $\mathscr{E}$ is reflexive.
  The author expects the result is known for torsion-free sheaves but could only find a reference over $\mathbb{P}^{n}$~\cite{okonek1980}*{Chapter 2, Lemma 1.3.3}.
  Either way, the argument is an easy corollary of Proposition \ref{proposition:boundOnSections} and we will use this result later.
    
  \begin{lemma}
    \label{lemma:torsionFreeDegreeZeroIsTrivial}
    Suppose $\mathscr{E}$ is a globally generated, torsion-free sheaf.
    If $\operatorname{deg}_{H}(\mathscr{E})=0$ then $\mathscr{E}\cong\mathscr{O}_{X}^{\oplus\operatorname{rank}(\mathscr{E})}$.
  \end{lemma}
  \begin{proof}
    By Proposition \ref{proposition:boundOnSections}, $h^{0}(\mathscr{E})\leq\operatorname{rank}(\mathscr{E})$.
    Therefore, since $\mathscr{E}$ is globally generated, there is a surjection $\mathscr{O}_{X}^{\oplus\operatorname{rank}(\mathscr{E})}\to\mathscr{E}\to 0$.
    By Lemma \ref{lemma:globallyGeneratedGivesInjection} this morphism is also injective.
    In other words, $\mathscr{E}=\mathscr{O}_{X}^{\oplus\operatorname{rank}(\mathscr{E})}$, as desired.
  \end{proof}
   
  \begin{remark}
    \label{remark:howGoodIsBound}
    We end this section by noting when Proposition \ref{proposition:boundOnSections} is close to optimal or not.
    We also note the assumptions of torsion-free and globally generated are both necessary.
    \begin{enumerate}
      \item{
        Suppose $\mathscr{L}$ is a torsion-free sheaf of rank $1$.
        By the Hirzebruch-Riemann-Roch theorem and Remark \ref{remark:rewritingBounds}, for $k\gg 0$
        \[
          B_{H}(\operatorname{deg}_{H}(\mathscr{L}(k)))-h^{0}(\mathscr{L}(k))\leq F(k)
        \] 
        where $F(k)$ is a polynomial in the variable $k$ of degree $n-2$.
        More generally, the same bound holds for coherent sheaves of the form $\mathscr{E}=\mathscr{O}_{X}^{\oplus\operatorname{rank}(\mathscr{E})-1}\oplus\mathscr{L}(k)$.
      }
      \item{
        As noted in Lemma \ref{lemma:torsionFreeDegreeZeroIsTrivial}, if $\operatorname{deg}_{H}(\mathscr{E})=0$ then the bound is optimal.
        
        Similarly for $X=\mathbb{P}^{n}$ our bound is optimal for all ranks and degrees.
        Specifically the vector bundle
        \[
          \mathscr{E}\cong\mathscr{O}_{X}^{\oplus r-1}\oplus\mathscr{O}_{X}(d)
        \]
        gives equality in Proposition \ref{proposition:boundOnSections}.
        
        Last, if $X$ is a smooth Del Pezzo surface then
        \[
          \mathscr{E}=\mathscr{O}_{X}^{\oplus r-1}\oplus\mathscr{O}_{X}(-dK_{X})
        \]
        also gives equality in Proposition \ref{proposition:boundOnSections}.
      }
      \item{
        If $\mathscr{E}$ is torsion-free (but not necessarily globally generated) then
        \[
          h^{0}(\mathscr{E})\leq\operatorname{rank}(\mathscr{E}) H^{n}\binom{\frac{\mu_{H}^{+}(\mathscr{E})}{H^{n}}+n-1+\sum_{i=1}^{\operatorname{rank}(\mathscr{E})}\frac{1}{i}}{n}.
        \]
        as shown by Langer~\cite{langer2004}*{Theorem 3.3}.
        If $\mu_{H}^{+}(\mathscr{E})-\mu_{H}(\mathscr{E})$ is small then Langer's bound is stronger than Proposition \ref{proposition:boundOnSections}.
        That is to say, Langer's bound is better when $\mathscr{E}$ is ``close" to being $\mu_{H}$-semistable while Proposition \ref{proposition:boundOnSections} is better when $\mathscr{E}$ is ``far" from being $\mu_{H}$-semistable.
        This is illustrated in the first two items of this remark.
        
        In either case, Langer's bound has a complicated dependence on $\operatorname{rank}(\mathscr{E})$ and $\mu_{H}^{+}(\mathscr{E})$, while the dependence on $\operatorname{rank}(\mathscr{E})$ in Proposition \ref{proposition:boundOnSections} is much simpler.
      }
      \item{
        If $\mathscr{E}$ is torsion-free (but not necessarily globally generated) on a smooth, integral curve of genus $g$ with $0\leq\mu_{H}^{-}(\mathscr{E})\leq\mu_{H}^{+}(\mathscr{E})\leq 2g-2$ then
        \[
          h^{0}(\mathscr{E})\leq \frac{\operatorname{deg}(\mathscr{E})}{2}+\operatorname{rank}(\mathscr{E}).
        \]
        To see this, induct on the length of a Harder-Narasimhan filtration of $\mathscr{E}$.
        The base case is due to Xiao~\cite{brambilla1997}*{Theorem 2.1}.
        
        This bound is almost the same as Lemma \ref{lemma:boundGlobalSectionsLineBundle} except the requirement $\mu_{H}^{+}(\mathscr{E})\leq 2g-2$ is weaker than the requirement $\operatorname{deg}_{H}(\mathscr{E})\leq 2g-2$.
        Either way, the dependence on $\mu_{H}^{+}(\mathscr{E})$ does not lend itself to the induction in Lemma \ref{lemma:boundGlobalSectionsLineBundle}.
      }
      \item{
        Proposition \ref{proposition:boundOnSections} is false for globally generated, torsion sheaves.
        For example, consider $\mathscr{O}_{Y}(d)$ as a coherent sheaf on $X$ where $\operatorname{codim}(Y)=1$.
        Then $\operatorname{rank}(\mathscr{O}_{Y}(d))=0$ and $\operatorname{deg}_{H}(\mathscr{O}_{Y}(d))=1$ but $h^{0}(\mathscr{O}_{Y}(d))=d$.
      }
      \item{
        Proposition \ref{proposition:boundOnSections} is false for torsion-free sheaves that are not globally generated.
        In fact, the number of global sections of an arbitrary torsion-free sheaf cannot be bounded solely in terms of topological invariants.
        For example, if $X=\mathbb{P}^{1}$ then the locally-free sheaves $\mathscr{O}_{C}(a)\oplus\mathscr{O}_{X}(-a)$ all have the same topological type \cite{okonek1980}*{Introduction of I.6.1}, so $h^{0}(\mathscr{E})$ cannot be bounded solely in terms of topological invariants of $\mathscr{E}$.
      }
    \end{enumerate}
  \end{remark}
  
\section{Slope Stability of Kernel Sheaves}  
  The following lemma shows if $0\to\mathscr{N}\to\mathscr{M}$ is a saturated subsheaf of a kernel sheaf $\mathscr{M}$ then $0>\operatorname{deg}(\mathscr{N})>\operatorname{deg}(\mathscr{M})$.
  
  \begin{lemma}
    Suppose $\mathscr{M}$ is a kernel sheaf associated to a globally generated, torsion-free $\mathscr{E}$.
    Suppose $0\to\mathscr{N}\to\mathscr{M}$ is a nonzero, proper saturated subsheaf satisfying $\mu_{H}(\mathscr{N})\geq\mu_{H}(\mathscr{M})$.
    \begin{enumerate}
      \label{lemma:degreesAsExpected}
      \item{
        The quotient $\mathscr{O}_{X}^{\oplus h^{0}(\mathscr{E})}/\mathscr{N}$ induced from the composition $\mathscr{N}\to\mathscr{M}\to\mathscr{O}_{X}^{\oplus h^{0}(\mathscr{E})}$ is torsion-free.
      }
      \item{
        $\operatorname{deg}_{H}(\mathscr{M})<\operatorname{deg}_{H}(\mathscr{N})<0$.
      }
    \end{enumerate}
  \end{lemma}
  \begin{proof}
    Suppose $0\to\mathscr{N}\to\mathscr{M}$ is a nonzero, proper, saturated subsheaf satisfying $\mu_{H}(\mathscr{N})\geq\mu_{H}(\mathscr{M})$.
    \begin{enumerate}
      \item{
        We have the following commutative diagram with exact rows and columns:
        \[
          \begin{tikzcd}
             & 0\arrow[d] & 0\arrow[d] &  &  \\
            0\arrow[r] & \mathscr{N}\arrow[r]\arrow[d] & \mathscr{O}_{X}^{\oplus h^{0}(\mathscr{E})}\arrow[r]\arrow[d] & \mathscr{O}_{X}^{\oplus h^{0}(\mathscr{E})}/\mathscr{N}\arrow[r]\arrow[d,dashed] & 0 \\
            0\arrow[r] & \mathscr{M}\arrow[r]\arrow[d] & \mathscr{O}_{X}^{\oplus h^{0}(\mathscr{E})}\arrow[r] \arrow[d]& \mathscr{E}\arrow[r] & 0 \\
              & \mathscr{M}/\mathscr{N}\arrow[d] & 0 &  & \\
              & 0 & &  & \\
          \end{tikzcd}
        \]
        where the dashed arrow exists by the universal property of the cokernel.
        Therefore, by the Snake Lemma, we obtain the short exact sequence
        \[
          0\to\mathscr{M}/\mathscr{N}\to\mathscr{O}_{X}^{\oplus h^{0}(\mathscr{E})}/\mathscr{N}\to\mathscr{E}\to 0.
        \]
        Since $\mathscr{N}$ is a saturated subsheaf of $\mathscr{M}$, by definition, $\mathscr{M}/\mathscr{N}$ is torsion-free.
        Since $\mathscr{M}/\mathscr{N}$ and $\mathscr{E}$ are torsion-free, $\mathscr{O}_{X}^{\oplus h^{0}(\mathscr{E})}/\mathscr{N}$ is also torsion-free, as claimed.
      }
      \item{
        There is an injection $0\to\mathscr{N}\to\mathscr{O}_{X}^{\oplus h^{0}(\mathscr{E})}$, so $\operatorname{deg}_{H}(\mathscr{N})\leq 0$.
        If $\operatorname{deg}_{H}(\mathscr{N})=0$, by part (1), $\mathscr{O}_{X}^{\oplus h^{0}(\mathscr{E})}/\mathscr{N}$ is a globally generated, torsion-free sheaf of degree $0$.
        Therefore, by Lemma \ref{lemma:torsionFreeDegreeZeroIsTrivial}, $\mathscr{O}_{X}^{\oplus h^{0}(\mathscr{E})}/\mathscr{N}=\mathscr{O}_{X}^{\oplus h^{0}(\mathscr{E})-\operatorname{rank}(\mathscr{N})}$.
        Taking cohomology gives the inequality
        \begin{equation}
          \label{equation:boundOnGlobalSections}
          h^{0}(\mathscr{E})\leq h^{0}(\mathscr{N})+(h^{0}(\mathscr{E})-\operatorname{rank}(\mathscr{N})).
        \end{equation}
        However, since $\mathscr{M}$ is a kernel sheaf, $H^{0}(\mathscr{M})=0$ and so $H^{0}(\mathscr{N})=0$ as well.
        Thus, Equation \ref{equation:boundOnGlobalSections} implies $\operatorname{rank}(\mathscr{N})=0$.
        However, since $\mathscr{N}$ is nonzero and $\mathscr{M}$ is torsion-free, this is not possible.
        Therefore, we must have $\operatorname{deg}_{H}(\mathscr{N})<0$, as claimed.
        
        Last, since $\mu_{H}(\mathscr{N})\geq\mu_{H}(\mathscr{M})$ and $\mathscr{M}/\mathscr{N}$ is torsion-free, by the seesaw inequality (Lemma \ref{lemma:seesawInequality}), $\mu_{H}(\mathscr{M}/\mathscr{N})\leq\mu_{H}(\mathscr{M})$.
        In particular, $\operatorname{deg}_{H}(\mathscr{M}/\mathscr{N})<0$ so $\operatorname{deg}_{H}(\mathscr{M})<\operatorname{deg}_{H}(\mathscr{N})$, as desired.
      }
    \end{enumerate}
  \end{proof}

    To prove Theorem \ref{maintheorem}, we first use Proposition \ref{proposition:boundOnSections} to show if $\mathscr{N}\to\mathscr{M}$ is a maximal destabilizing subsheaf then $\mu_{H}(\mathscr{N})$ is bounded in terms of invariants of $\mathscr{N}$, $H$, and $X$.
  We then aim to use Lemma \ref{lemma:degreesAsExpected} to bound $\mu_{H}(\mathscr{N})$ in terms of invariants of $\mathscr{M}$ and $X$.
  The following lemma is a technical step needed to achieve this bound.
  See Definition \ref{definition:notationForBounds} for a reminder of the functions $A_{H}(d)$ and $B_{H}(d)$.
  
  \begin{lemma}
    \label{lemma:boundsAreIncreasing}
    The function $\frac{-d}{A_{H}(d)}$ is increasing for $d\in (0,\infty)$.
    Similarly, the function $\frac{-d}{B_{H}(d)}$ is increasing for $d\in (2g-2,\infty)$.
  \end{lemma}
  \begin{proof}
    By definition $A_{H}(d)/d$ is a polynomial in $d$ whose coefficients are all nonnegative and the constant coefficient is $0$.
    Therefore, $A_{H}(d)/d$ is increasing for $d\in (0,\infty)$.
    It follows that $-d/A_{H}(d)$ is also increasing for $d\in (0,\infty)$.
    
    We now consider $B_{H}(d)$.
    For $n\geq 3$,
    \[
      B_{H}(d)=F\big{(}d-(g-1)\big{)}+G\big{(}d-(2g-2)\big{)}-1
    \]
    where $F$ (resp. $G$) is a single variable polynomial of degree $n$ (resp. $n-2$) whose coefficients are all nonnegative.
    Therefore, for $d>2g-2$, $B_{H}(d)/d$ is increasing.
    For $n=1,2$ the result can be checked by directly calculating the derivative of $B_{H}(d)/d$.
    It follows that $-d/B_{H}(d)$ is also increasing for $d>2g-2$, as desired.
  \end{proof}
  
  \begin{theorem}
    \label{theorem:lazarsfeldMukaiSheavesAreStable}
    Assume $X$ is a smooth, projective variety with very ample divisor $H$.
    Let $g$ be the sectional genus of $X$ with respect to $H$.
    Suppose $\mathscr{L}$ is a globally generated, torsion-free sheaf of rank $1$ on $X$ with associated kernel sheaf $\mathscr{M}$.
    If
    \[
      h^{0}(\mathscr{L})-1>\max\left\{\frac{\operatorname{deg}_{H}(\mathscr{L})}{2g-2}A_{H}(2g-2),\frac{\operatorname{deg}_{H}(\mathscr{L})}{\operatorname{deg}_{H}(\mathscr{L})-1}B_{H}(\operatorname{deg}_{H}(\mathscr{L})-1)\right\}
    \]
    then $\mathscr{M}$ is $\mu_{H}$-stable.
  \end{theorem}
  \begin{proof}
    We proceed by contradiction, so suppose $0\to\mathscr{N}\to\mathscr{M}$ is nonzero, proper saturated subsheaf satisfying $\mu_{H}(\mathscr{N})\geq\mu_{H}(\mathscr{M})$.
    The composition $\mathscr{N}\to\mathscr{M}\to\mathscr{O}_{X}^{\oplus h^{0}(\mathscr{L})}$ induces the short exact sequence
    \[
      0\to\mathscr{N}\to\mathscr{O}_{X}^{\oplus h^{0}(\mathscr{L})}\to\mathscr{C}\to 0.
    \]
    By Lemma \ref{lemma:degreesAsExpected}.1, $\mathscr{C}$ is torsion-free.    
    Furthermore, since $\mathscr{M}$ is a kernel sheaf, $H^{0}(\mathscr{M})=0$ and so $H^{0}(\mathscr{N})=0$.
    Therefore, taking cohomology of the above short exact sequence shows $h^{0}(\mathscr{L})\leq h^{0}(\mathscr{C})$.
    In other words,
    \[
      \operatorname{rank}(\mathscr{N})=h^{0}(\mathscr{L})-\operatorname{rank}(\mathscr{C})\leq h^{0}(\mathscr{C})-\operatorname{rank}(\mathscr{C}).
    \]
    
    Since $\mathscr{C}$ is a globally generated, torsion-free sheaf on $X$, by Proposition \ref{proposition:boundOnSections}, if $\operatorname{deg}_{H}(\mathscr{C})\leq 2g-2$ then $\operatorname{rank}(\mathscr{N})\leq A_{H}(\operatorname{deg}_{H}(\mathscr{C}))$ and if $\operatorname{deg}_{H}(\mathscr{C})>2g-2$ then $\operatorname{rank}(\mathscr{N})\leq B_{H}(\operatorname{deg}_{H}(\mathscr{C}))$.
    Furthermore, by (2) of Lemma \ref{lemma:degreesAsExpected},
    \[
      \operatorname{deg}_{H}(\mathscr{C})=-\operatorname{deg}_{H}(\mathscr{N})<-\operatorname{deg}_{H}(\mathscr{M})=\operatorname{deg}_{H}(\mathscr{L}).
    \]
    Hence, by Lemma \ref{lemma:boundsAreIncreasing}, if $\operatorname{deg}_{H}(\mathscr{C})\leq 2g-2$ then
    \[
      \mu_{H}(\mathscr{N})=\frac{-\operatorname{deg}_{H}(\mathscr{C})}{\operatorname{rank}(\mathscr{N})}\leq \frac{-\operatorname{deg}_{H}(\mathscr{C})}{A_{H}(\operatorname{deg}_{H}(\mathscr{C}))}\leq \frac{-(2g-2)}{A_{H}(2g-2)},
    \]
    and if $\operatorname{deg}_{H}(\mathscr{C})>2g-2$ then
    \[
      \mu_{H}(\mathscr{N})=\frac{-\operatorname{deg}_{H}(\mathscr{C})}{\operatorname{rank}(\mathscr{N})}\leq \frac{-\operatorname{deg}_{H}(\mathscr{C})}{B_{H}(\operatorname{deg}_{H}(\mathscr{C}))}\leq \frac{-(\operatorname{deg}_{H}(\mathscr{L})-1)}{B_{H}(\operatorname{deg}_{H}(\mathscr{L})-1)}.
    \]
 
    By the assumed bound on $h^{0}(\mathscr{L})-1$, if $\operatorname{deg}_{H}(\mathscr{C})\leq 2g-2$ then
    \[
      \mu_{H}(\mathscr{N})\leq \frac{-(2g-2)}{A_{H}(2g-2))}< \frac{-\operatorname{deg}_{H}(\mathscr{L})}{h^{0}(\mathscr{L})-1}=\mu_{H}(\mathscr{M}),
    \]
    and if $\operatorname{deg}_{H}(\mathscr{C})>2g-2$ then
    \[
      \mu_{H}(\mathscr{N})\leq \frac{-(\operatorname{deg}_{H}(\mathscr{L})-1)}{B_{H}(\operatorname{deg}_{H}(\mathscr{L})-1))}< \frac{-\operatorname{deg}_{H}(\mathscr{L})}{h^{0}(\mathscr{L})-1}=\mu_{H}(\mathscr{M}).
    \]
    Since $\mu_{H}(\mathscr{N})\geq\mu_{H}(\mathscr{M})$ we have reached a contradiction!
    Hence, $\mathscr{M}$ is $\mu_{H}$-stable, as desired.
  \end{proof}
  
  It is clear from the argument that if we have equality rather than inequality in either of the bounds in Theorem \ref{theorem:lazarsfeldMukaiSheavesAreStable} then $\mathscr{M}$ is only $\mu_{H}$-semistable.
  
  We show the bounds of Theorem \ref{theorem:lazarsfeldMukaiSheavesAreStable} apply for any sufficiently positive twist.
  This gives a new proof of \cite{ein2013}*{Theorem A} and generalizes \cite{ein2013}*{Proposition C} to the case of arbitrary Picard group (which proves \cite{ein2013}*{Conjecture 2.6}).
  
  \begin{corollary}
    \label{corollary:assymptoticVersion}
    With the assumptions of Theorem \ref{theorem:lazarsfeldMukaiSheavesAreStable}, let $\mathscr{M}_{k}$ be the kernel sheaf associated to $\mathscr{L}\otimes\mathscr{O}_{X}(kH)$.
    If $\operatorname{dim}(X)\geq 2$ and $k\gg 0$ then $\mathscr{M}_{k}$ is $\mu_{H}$-stable.
  \end{corollary}
  \begin{proof}
    By Serre's theorem, if $k\gg 0$ then $H^{i}(\mathscr{L}(k))=0$ for all $i>0$.
    Therefore, by the Hirzebruch-Riemann-Roch theorem,
    \[
      h^{0}(\mathscr{L}(k))-1=\frac{H^{n} k^{n}}{n!}+\frac{\operatorname{deg}_{H}(\mathscr{L})+\frac{\operatorname{c}_{1}(X)\cdot H^{n-1}}{2}}{(n-1)!}k^{n-1}+\cdots
    \]
    for all $k\gg 0$ where the unwritten term is a polynomial in $k$ of degree $k^{n-2}$.
    
    Since
    \[
      \frac{\operatorname{deg}_{H}(\mathscr{L}(k))}{2g-2}\frac{2g-2+2n}{2n}\binom{\frac{2g-2}{H^{n}}+n-1}{n-1}=\frac{\operatorname{deg}_{H}(\mathscr{L})+k H^{n}}{2g-2}\frac{2g-2+2n}{2n}\binom{\frac{2g-2}{H^{n}}+n-1}{n-1}
    \]
    grows linearly in $k$, the first bound of Theorem \ref{theorem:lazarsfeldMukaiSheavesAreStable} is satisfied for $k\gg 0$.
    On the other hand, by Remark \ref{remark:rewritingBounds},
    \begin{align*}
      & \frac{\operatorname{deg}_{H}(\mathscr{L}(k))}{\operatorname{deg}_{H}(\mathscr{L}(k))-1} B_{H}(\operatorname{deg}_{H}(\mathscr{L}(k))-1) \\
      & \qquad =\left(1+\frac{1}{\operatorname{deg}_{H}(\mathscr{L})+k H^{n}-1}\right)\left(H^{n}\frac{k^{n}}{n!}+\left(\operatorname{deg}_{H}(\mathscr{L})-1+\frac{\operatorname{c}_{1}(X)\cdot H^{n-1}}{2}\right)\frac{k^{n-1}}{(n-1)!}+\cdots\right) \\
      & = H^{n} \frac{k^{n}}{n!}+\left(\operatorname{deg}_{H}(\mathscr{L})+\frac{c_{1}(X)\cdot H^{n-1}}{2}+\frac{1}{n}-1\right)\frac{k^{n-1}}{(n-1)!}+\cdots+\frac{\alpha}{\operatorname{deg}_{H}(\mathscr{L})+k H^{n}-1}
    \end{align*}
    where $\alpha$ is a constant independent of $k$ and the unwritten terms are polynomials in $k$ of smaller degree.
    Therefore, since $n\geq 2$, we find that the second bound of Theorem \ref{theorem:lazarsfeldMukaiSheavesAreStable} is satisfied for all $k\gg 0$.
    Therefore, $\mathscr{M}_{k}$ is $\mu_{H}$-stable for all $k\gg 0$, as desired.
  \end{proof}
  
  We remark on how to find explicit bounds for Corollary \ref{corollary:assymptoticVersion}.
  In practice, the arithmetic is too finicky to do by hand, but it is easy using a computer algebra program.
  This solves \cite{ein2013}*{Problem 2.4}.
  
  \begin{remark}
    \label{remark:howToGetExplicitBound}
    We show there is an effective bound on $k\gg 0$ in Corollary \ref{corollary:assymptoticVersion} only involving the Castelnuovo-Mumford regularity of $\mathscr{L}$ (see \cite{lazarsfeld2004}*{Definition 1.8.4}) and topological invariants of $\mathscr{L}$, $H$, and $X$.
    
    Suppose $k\geq\operatorname{reg}_{H}(\mathscr{L})$---the Castelnuovo-Mumford regularity of $\mathscr{L}$ with respect to $H$.
    By Mumford's theorem, $H^{i}(\mathscr{L}(k))=0$ for all $i>0$ and $\mathscr{L}(k)$ is globally generated.
    In this case, by the Hirzebruch-Riemann-Roch theorem, $h^{0}(\mathscr{L}(k))(\operatorname{deg}_{H}(\mathscr{L}(k))-1)$ is a polynomial of degree $k^{n+1}$.
    
    On the other hand, the required bound of Theorem \ref{theorem:lazarsfeldMukaiSheavesAreStable}, can be rewritten as
    \[
      (h^{0}(\mathscr{L}(k))-1)(2g-2)>\operatorname{deg}_{H}(\mathscr{L}(k))A_{H}(2g-2).
    \]
    and
    \[
      (h^{0}(\mathscr{L}(k))-1)(\operatorname{deg}_{H}(\mathscr{L}(k))-1)>\operatorname{deg}_{H}(\mathscr{L}(k))B_{H}(\operatorname{deg}_{H}(\mathscr{L}(k)-1).
    \]
    Since $h^{0}(\mathscr{L}(k))(\operatorname{deg}_{H}(\mathscr{L}(k))-1)$ is a polynomial in $k$, finding explicit $k$ such that Corollary \ref{corollary:assymptoticVersion} holds is equivalent to finding the largest real zero of the degree $n+1$ polynomial associated with the above inequalities.
    For small $n$ this can be found exactly.
    For any $n$, such a bound on $k$ can be found using Cauchy's bound on real zeros.
  \end{remark}
  
  Theorem \ref{theorem:lazarsfeldMukaiSheavesAreStable} does \textit{not} actually use the assumption $\operatorname{rank}(\mathscr{L})=1$.
  However, the author is unable to give an example of a higher rank sheaf satisfying the bound of Theorem \ref{theorem:lazarsfeldMukaiSheavesAreStable}.
  In fact, by Remark \ref{remark:rewritingBounds},
  \[
    B_{H}(\operatorname{deg}_{H}(\mathscr{E}(k)))= H^{n}\frac{\operatorname{rank}(\mathscr{E})^{n} k^{n}}{n!}+\cdots,
  \]
  while
  \[
    h^{0}(\mathscr{E}(k))=H^{n} \frac{\operatorname{rank}(\mathscr{E} )k^{n}}{n!}+\cdots
  \]
  for all $k\gg 0$.
  In other words, the argument of Corollary \ref{corollary:assymptoticVersion} fantastically fails when $\operatorname{rank}(\mathscr{E})\geq 2$ and $\operatorname{dim}(X)\geq 2$.
    
  With this in mind, it is natural to ask whether Theorem \ref{theorem:lazarsfeldMukaiSheavesAreStable} extends to higher ranks.
  In other words, the kernel sheaf associated to a sufficiently positive, globally generated, torsion-free, $\mu_{H}$-stable sheaf also $\mu_{H}$-stable?
  As noted in the introduction, \cite{butler1994}*{Theorem 1.2} gives such a result for curves.
  As far as the author knows, there are no such results in higher dimensions.
  Furthermore, as stated in Remark \ref{remark:howGoodIsBound}, the bound on $h^{0}(\mathscr{E})$ from Proposition \ref{proposition:boundOnSections} is close to optimal.
  For these reasons, if we were to try to generalize the method of Theorem \ref{theorem:lazarsfeldMukaiSheavesAreStable} to higher ranks, we would need more control over quotient sheaf $\mathscr{O}_{X}^{\oplus h^{0}(\mathscr{E})}/\mathscr{N}$.
  Morally, an improved analysis of this quotient is how Butler is able to obtain a higher rank result for curves~\cite{butler1994}*{Lemma 1.9}.
  A similar analysis is done in \cite{trivedi2010} in the case of line bundles on $\mathbb{P}^{n}$.
    
  Either way, in an upcoming article, the author gives partial results for stability of kernel sheaves associated to a $\mu_{H}$-stable, globally generated, torsion-free sheaves (of arbitrary rank) on Del Pezzo surfaces.
  That article uses a completely different method based on Bridgeland stability.

\end{document}